\documentclass[11pt]{amsart}

\usepackage[margin=1.2in]{geometry}
\usepackage{amsmath,amssymb,comment,amsthm,url} \usepackage[inline]{enumitem}

\usepackage{color}

\usepackage{graphicx}

\title{A generalization of the practical numbers}
\renewcommand\rightmark{\centerline{\normalsize\sc
A generalization of the practical numbers}\thepage}

\author{Nicholas Schwab}
\address{Department of Mathematics\\ Universit\"{a}t Bonn\\ Bonn, Germany 53115}
\email{nicholas.schwab@uni-bonn.de}

\author{Lola Thompson}
\address{Department of Mathematics\\ Oberlin College\\ Oberlin, OH USA 44074}
\email{lola.thompson@oberlin.edu}

\DeclareMathAlphabet{\curly}{U}{rsfs}{m}{n}

\DeclareSymbolFont{bbold}{U}{bbold}{m}{n}
\DeclareSymbolFontAlphabet{\mathbbold}{bbold}

\newtheorem{thm}{Theorem}[section]
\newtheorem{lemma}[thm]{Lemma}
\newtheorem{cor}[thm]{Corollary}
\newtheorem{prop}[thm]{Proposition}
\newtheorem{definition}{Definition}

\newtheorem{example}{Example}
\newtheorem*{theorem*}{Theorem}
\theoremstyle{remark}\newtheorem*{remark}{Remark}

\numberwithin{equation}{section}

\newcommand\Z{\mathbf{Z}}

\newcommand\N{\mathbf{N}}

\renewcommand{\phi}{\varphi}
\makeatletter
\renewcommand{\pod}[1]{\mathchoice
  {\allowbreak \if@display \mkern 18mu\else \mkern 8mu\fi (#1)}
  {\allowbreak \if@display \mkern 18mu\else \mkern 8mu\fi (#1)}
  {\mkern4mu(#1)}
  {\mkern4mu(#1)}
}

\begin{document} 

\begin{abstract} 
 A positive integer $n$ is practical if every $m \leq n$ can be written as a sum of distinct divisors of $n$. One can generalize the concept of practical numbers by applying an arithmetic function $f$ to each of the divisors of $n$ and asking whether all integers in a given interval can be expressed as sums of $f(d)$'s, where the $d$'s are distinct divisors of $n$. We will refer to such $n$ as `$f$-practical.' In this paper, we introduce the $f$-practical numbers for the first time. We give criteria for when all $f$-practical numbers can be constructed via a simple necessary-and-sufficient condition, demonstrate that it is possible to construct $f$-practical sets with any asymptotic density, and prove a series of results related to the distribution of $f$-practical numbers for many well-known arithmetic functions $f$.
\end{abstract}


\maketitle

\section{Introduction}
Srinivasan first introduced the practical numbers as integers $n$ for which every number between $1$ and $n$ is representable as a sum of distinct divisors of $n$. In her Ph.D. thesis, the second author adapted this concept to study the degrees of divisors of $x^n-1$. Recall that $x^n - 1 = \prod_{d \mid n} \Phi_d(x)$ where $\Phi_d(x)$ is the $d^{th}$ cyclotomic polynomial with $\deg \Phi_d(x) = \varphi(d)$. By applying Euler's totient function on the divisors of $n$, the second author categorized the numbers $n$ for which $x^n-1$ has a divisor in $\Z[x]$ of every degree smaller than $n$, calling these integers ``$\varphi$-practical.'' The aim of the present paper is to generalize much of the existing literature on practical and $\varphi$-practical numbers.


\begin{definition}
Let $f:\N\to \N$ be a multiplicative function. We define
\begin{align*}
    S_f(n) = \sum_{d\mid n} f(d).
\end{align*}
Therefore we have $S_f = f\ast \mathbbold{1}$, where $\mathbbold{1}(n)$ denotes the arithmetic function that is identically $1$.
\end{definition}

We note that the function $S_f(n)$ is multiplicative, since it is the Dirichlet convolution of two multiplicative functions.

\begin{definition}\label{def:fpractical}
Let $f : \N \to \N$ be a multiplicative function for which for every prime $p$ and every positive integer $k$ satisfies $f(p^{k-1}) \leq f(p^k)$. A positive integer $n$ is called $f$-practical if for every positive integer $m\leq S_f(n)$ there is a set $\mathcal{D}$ of divisors of $n$ for which 
\begin{align*}
    m = \sum_{d\in \mathcal{D}} f(d)
\end{align*}
holds. \end{definition}


\begin{example} If $f = I$ (the identity function), this is equivalent to the definition of practical. \end{example}

\begin{example} If $f = \varphi$, this is precisely the definition of $\varphi$-practical. \end{example}

One can check whether all integers in an interval can be expressed as subsums of numbers from a particular set by applying the following naive algorithm (cf. \cite[Theorem A.1]{thesis}):

\begin{prop}\label{thm:thesisA1}
Let $w_1\leq w_2\leq\dots\leq w_k$ be positive integers with $\sum_{i=1}^k w_i = s$. Then, every integer in $[0,s]$ can be represented as a sum of some subset of the $w_i$'s if and only if $w_{i+1} \leq 1 + w_1 +\dots +w_i$ holds for every $i<k$. 
\end{prop} 

\noindent It is not so ``practical'' to use this algorithm to determine whether an integer $n$ is practical. Instead, it is useful to have a criterion in terms of the prime factors of $n$. In \cite{Stewart} Stewart gave such a criterion for constructing practical numbers and proved that every practical number can be obtained in this fashion. Stewart's criterion for a number to be practical serves as a key lemma in many subsequent papers on the practical numbers. The second author showed in \cite{thesis} that there are $\phi$-practical numbers which cannot be constructed in such a manner. In the present paper, we examine the functions for which this means of constructing $f$-practical numbers is possible and obtain the following theorem.

\begin{thm}\label{thm:stewartlike}
Let $f$ be a multiplicative function. All $f$-practical numbers are constructable by a Stewart-like criterion if and only if for every prime $p$ for which there is a coprime integer $n$ with $f(p) \leq 1+\sum_{d\mid n}f(d)$ the inequality
\begin{align*}
    f(p^{k+1})\leq f(p)f(p^k)
\end{align*}
holds for every integer $k\geq 0$.
\end{thm} 

One of the aims of this paper is to study the distribution of $f$-practical numbers for various arithmetic functions $f$. We show that it is possible to construct $f$-practical sets with any asymptotic density. In fact:

\begin{thm}\label{thm:densitiesaredense} 
The densities of the $f$-practical sets are dense in $[0,1].$
\end{thm}

The $f$-practical sets that are the most interesting to study are those that are neither finite nor all of $\N$. Intuitively, if the values of $f(n)$ are too large relative to $n$, then some integers in the interval $[1, S_f(n)]$ will always be skipped, resulting in a finite set of $f$-practical numbers. On the other hand, if the values of $f(n)$ are too small relative to $n$, then every integer winds up being $f$-practical. Thus, the arithmetic functions which produce non-trivial $f$-practical sets are those that behave like the identity function, i.e., those for which $f(p^k) \approx (f(p))^k$ at all prime powers. Examples of functions satisfying this condition include $I$, $\varphi$, $\lambda$ (the Carmichael $\lambda$-function), and $\varphi^*$ (the unitary totient function).





\begin{thm}\label{thm:fpracticalchebyshev}
Let $f = \phi^*$. Let $F_{f}(X)$ be the number of $f$-practical numbers less than or equal to $X$. Then there exist positive constants $l_{f}$ and $u_{f}$ such that
\begin{align*}
    l_{f}\frac{X}{\log X}\leq F_{f}(X)\leq u_{f}\frac{X}{\log X}
\end{align*}
for all $X\geq2$.
\end{thm}

The result also holds for $f = I$ and $f = \varphi$. For these functions, the distributions have been well-studied. In a 1950 paper, Erd\H{o}s \cite{Erdos} claimed that the practical numbers have asymptotic density $0$. Subsequent papers by Hausman and Shapiro \cite{H&S}, Margenstern \cite{Margenstern}, Tenenbaum \cite{Tenenbaum}, and Saias \cite{Saias} led to sharp upper and lower bounds for the count of practical numbers in the interval $[1, X]$. A recent paper of Weingartner \cite{Weingartner} showed that the count of practical numbers is asymptotically $c X/\log X$, for some positive constant $c$. In her PhD thesis \cite{thesis}, the second author proved sharp upper and lower bounds for the count of $\varphi$-practical numbers. This work was improved to an asymptotic in a subsequent paper with Pomerance and Weingartner \cite{ptw}.

For $f = \lambda$, computational evidence seems to suggest that $X/\log X$ is the correct order of magnitude for the $f$-practicals. Indeed, we prove that the upper bound from Theorem \ref{thm:fpracticalchebyshev} holds when $f = \lambda$. However, we have not been able to obtain a sharp lower bound.

The proofs of the aforementioned theorems rely heavily on the fact that the functions $f$ are multiplicative (or nearly multiplicative, in the case of the Carmichael $\lambda$-function). We are also able to prove density results for certain non-multiplicative $f$. For example, we classify the additive functions $f$ for which all positive integers are $f$-practical. We also examine some $f$-practical sets where $f$ is neither additive nor multiplicative. 

The paper is organized according to the following scheme. In Section \ref{f-practicalcriterion}, we provide a method for constructing infinite families of $f$-practical numbers. In Section \ref{stewart-like}, we classify the set of all $f$-practical numbers that can be completely determined via a Stewart-like condition on the sizes of the prime factors. In Section \ref{densities}, we give examples of $f$-practical sets with various densities and show that the densities themselves are dense in $[0,1]$. In Section \ref{boundsforcertainf}, we prove upper and lower bounds for the sizes of the sets of $f$-practical numbers for certain arithmetic functions $f$. Much of the work in the aforementioned sections applies only to functions which are multiplicative or nearly multiplicative. In Section \ref{non-multiplicativef}, we give density results for $f$-practicals for some well-known non-multiplicative functions $f$.

Throughout this paper, we will use $n$ to denote an integer and $p$ to denote a prime number. Moreover, we will use $P(n)$ to represent the largest prime factor of $n$. 


\section{$f$-practical construction for multiplicative $f$}\label{f-practicalcriterion}


In this section, we develop the basic machinery for constructing infinite families of $f$-practical numbers. Following the definition of weakly $\phi$-practical numbers (see \cite[Definition 4.4]{Thompson}), we introduce the concept of weakly $f$-practical numbers.

\begin{definition}\label{def:weaklyfpractical}
Let $n = p_1^{e_1} \dots p_k^{e_k}$, where $f(p_1) \leq f(p_2) \leq \cdots \leq f(p_k)$. We define $m_i = \prod_{j=1}^i p_j^{e_j}$ for every non-negative integer $i<k$. We call $n$ weakly $f$-practical if for every $i$ 
\begin{align*}
    f(p_{i+1}) \leq S_f(m_i) + 1
\end{align*}
holds. 
\end{definition}


\begin{thm}\label{thm:f-practicalisweakly}
Every $f$-practical number is also weakly $f$-practical.
\end{thm}

\begin{proof}
Let $n = \prod_{i=1}^k p_i^{e_i}$ be $f$-practical with $f(p_1) \leq f(p_2) \leq \dots \leq f(p_k)$. Now let $m_i = \prod_{j=1}^ip_j^{e_j}$ for every $i<k$. Suppose $n$ is not weakly $f$-practical. Then there must be an $i<k$ so $f(p_{i+1}) > S_f(m_i) +1$ holds. For every $j>i+1$ we have $f(p_j) \geq f(p_{i+1})$. Therefore for every divisor $d$ of $n$ which does not divide $m_i$ we have $f(d)\geq f(p_{i+1})$ since $f$ is multiplicative and $d$ must be divisible by some $p_j$ with $j>i$. Because the sum of $f(t)$ over all $t\mid m_i$ is exactly $S_f(m_i)$ there is no possibility to express $S_f(m_i) +1$ as a sum of $f(d)$ for $d\mid n$. This contradicts the fact that $n$ is $f$-practical.
\end{proof}
\begin{remark}
In this proof we did not use the condition that $f(ab) = f(a)f(b)$ but the implied and weaker fact that $f(ab) \leq f(a)f(b)$. Hence this proof also holds, for example, for the Carmichael $\lambda$-function. In fact, since $\lambda(a) \leq \phi(a)$ for any integer $a$ and $\lambda(p) = \phi(p)$, every weakly $\lambda$-practical number is also weakly $\phi$-practical, because $\phi(p_{i+1}) = \lambda(p_{i+1}) \leq  S_\lambda(m_i) +1 \leq S_\phi(m_i)+1$.
\end{remark}

\begin{cor}\label{cor:smallprime}
If $n$ is weakly $f$-practical and $p\leq P(n)$ then $pn$ is also weakly $f$-practical. 
\end{cor}

The following theorem gives a necessary and sufficient condition for a product of a prime power and an $f$-practical number to be $f$-practical itself. For some functions $f$ (like the identity function) this gives a way to construct all $f$-practical numbers (cf. \cite[Corollary 1]{Stewart}). This is not the case for all functions. For example, for the $\phi$ function, there are numbers that are $\phi$-practical that are not the product of a $\phi$-practical number and a prime power, as is the case with $45=3^2\cdot5$.

\begin{thm}\label{thm:f-practicalconstruct}
Let $n$ be $f$-practical. Let $p$ be prime with $(p,n) = 1$. Then $np^k$ is $f$-practical if and only if $f(p^i) \leq S_f(np^{i-1}) +1$ for all $1\leq i\leq k$.
\end{thm}

\begin{proof}
If $f(p^i)> S_f(np^{i-1}) +1$ for some $i\geq 1$, then $S_f(np^{i-1})+1$ is not representable as a sum of $f(d)$'s with $d\mid n$, since for every divisor $d$ of $np^k$ which does not divide $np^{i-1}$ we have $f(d)> S_f(np^{i-1})+1$.

We show by induction on $k$ that $np^{k}$ is $f$-practical for all $k$. Assume $np^{k-1}$ is $f$-practical for some $k\geq 1$. For $k=1$ we have $np^{k-1}=n$ which is given to be $f$-practical. We examine the intervals between $af(p^k)$ and $af(p^k) +S_f(np^{k-1})$ for $a = 0,1,\dots, S_f(n)$. Since 
\begin{align*}
    (a+1) f(p^k) \leq af(p^k) + S_f(np^{k-1}) +1
\end{align*}
holds the intervals overlap or are contiguous. Because we have
\begin{align*}
    S_f(n)f(p^k) + S_f(np^{k-1}) = \sum_{d\mid n}f(dp^k) + \sum_{d\mid np^{k-1}}f(d) = S_f(np^k),
\end{align*}
1 and $S_f(np^k)$ are included in these intervals. Thus, every integer $m$ with $1\leq m\leq S_f(np^k)$ is representable as $af(p^k) + b$ with $0\leq a \leq S_f(n)$ and $0\leq b\leq S_f(np^{k-1})$. As a result, we have two sets $\mathcal{D}$ and $\mathcal{T}$ of divisors of $n$ and $np^{k-1}$ respectively, so
\begin{align*}
    a = \sum_{d\in\mathcal{D}}f(d) && b = \sum_{t\in \mathcal{T}}f(t).
\end{align*}
Therefore we have
\begin{align*}
    m = f(p^k) \sum_{d\in\mathcal{D}} f(d)+ \sum_{t\in \mathcal{T}} f(t)= \sum_{d\in\mathcal{D}} f(dp^k)+ \sum_{t\in \mathcal{T}} f(t).
\end{align*}
Since $t\mid np^{k-1}$ and $p^k\mid dp^k$ the $t$'s and $dp^k$'s are distinct. Hence, if we take $\mathcal{E} = \{dp^k: d\in \mathcal{D}\} \cup\mathcal{T}$, we can write
\begin{align*}
    m= \sum_{e\in \mathcal{E}}f(e)
\end{align*}
where every $e$ divides $np^k$. Therefore $np^k$ is $f$-practical. 
\end{proof}



\begin{cor}
Every squarefree integer is $f$-practical if and only if it is weakly $f$-practical.
\end{cor}

\begin{proof}
By Theorem \ref{thm:f-practicalisweakly}, every squarefree $f$-practical number is also weakly $f$-practical. 

Let $n= p_1\dots p_k$ be weakly $f$-practical and squarefree. For every $i=0,1,\dots, k$ we define $m_i=p_1\dots p_i$. Since $n$ is weakly $f$-practical, we have $f(p_{i+1}) \leq S_f(m_i) +1$ for every $i\geq k$. Since 1 is $f$-practical we get from Theorem \ref{thm:f-practicalconstruct} that every $m_i$ is $f$-practical and $n$ is also $f$-practical.
\end{proof}

\begin{thm}\label{thm:everyint}
Every integer $n\in \N$ is $f$-practical if and only if \begin{align}\label{eq:everyint}
f(p^{k}) \leq S_f(p^{k-1})+1
\end{align}
holds for every prime $p$ and every integer $k\geq1$. 
\end{thm}

\begin{proof}
If the inequality holds we can use the fact that $S_f$ is multiplicative to show that
\begin{align*}
    S_f(np^{k-1}) +1  = S_f(n)S_f(p^{k-1}) +1 \geq S_f(n)(f(p^k)-1) +1 
\end{align*}
for every integer $n$ coprime to $p$. Furthermore for $f(p^k)\not=1$ the inequality $S_f(n)(f(p^k)-1)+1\geq f(p^k)$ is equivalent by multiplication by $f(p^k)-1$ to $S_f(n) \geq 1$ which holds for every $n$ since $f(1) = 1$ and $f(m)\geq 0$ for every $m\in\N$. In addition for $f(p^k) = 1$ we have $S_f(n)(f(p^k)-1) +1 = 1 =f(p^k)$. Hence the inequality

\begin{align*}
    S_f(np^{k-1}) +1 \geq S_f(n)(f(p^k)-1) +1 \geq f(p^k)
\end{align*}
holds for every prime $p$, any integer $k\geq 1$ and any integer $n\geq0$ coprime to $p$. Thereby the condition for Theorem \ref{thm:f-practicalconstruct} holds for any $n$. Because $1$ is always $f$-practical we can construct every integer greater than 1 as a product of prime powers which are $f$-practical which implies that this integer is also $f$-practical, if $f$ satisfies \eqref{eq:everyint} for every $p$ and $k$.

If there exists a $k\geq 1$ so $f(p^{k}) > S_f(p^{k-1})+1$ holds, $p^k$ is not $f$-practical since $S_f(p^{k-1})+1\leq S_f(p^k)$ is not representable as a sum of $f(d)$'s for some $d\mid p^k$. Hence not every integer greater than $0$ is $f$-practical.
\end{proof}

\section{Classifying functions with Stewart-like criteria}\label{stewart-like}

Stewart gave a way to construct every practical number as a product of practical numbers and prime powers (\cite[Corollary 1]{Stewart}). As shown in the previous section, this is not possible for the $\phi$-practical numbers. We will now categorize the functions for which this means of construction is possible.

\begin{definition}\label{def:iffconvenient}
A function $f$ is \textbf{convenient} if and only if every weakly $f$-practical number is also $f$-practical. 
\end{definition}

The following theorem gives an explicit way to check whether a function $f$ is convenient.

\begin{thm}
Let $P_f$ be the set of the prime numbers which are $f$-practical. It is easy to see, that these are exactly the primes $p$ with $f(p)\leq 2$. Then $f$ is convenient if an integer $n$ is $f$-practical if and only if $n$ is of the form
\begin{align*}
    n = p_1^{a_1}\dots p_k^{a_k}q_1^{b_1}\dots q_l^{e_l}
\end{align*}
 with primes $p_1,\dots,p_k\in P_f$ and $q_1,\dots,p_l\not\in P_f$ with exponents $a_1,\dots, a_k \geq 1$ and $b_1,\dots, b_l \geq 1$ (respectively) satisfying the following conditions
\begin{align*}
    f(p_1)\leq \dots \leq f(p_k)<f(q_1)\leq \dots \leq f(q_l) && k>0 \\
    f(q_{i+1}) \leq S_f(p_1^{a_1}\dots p_k^{a_k}q_1^{b_1}\dots q_i^{b_i}) +1 &&\text{ for } i=0,1,\dots, l-1.
\end{align*}
\end{thm}\label{def:convenient}

\begin{proof} We have shown in Theorem \ref{thm:f-practicalisweakly} that, for every $f$, each $f$-practical number is also weakly $f$-practical. Hence a function $f$ is convenient if and only if the set of weakly $f$-practical numbers is identical to the set of $f$-practical numbers. \end{proof}

Stewart's condition shows that the identity function is convenient, whereas $\phi$ is inconvenient as the number $75=3 \cdot 5^2$ fulfills every condition and thereby is weakly $\varphi$-practical but not $\phi$-practical. 

\begin{thm}
A function is convenient if and only if for every prime $p$ and $f$-practical integer $m$ which is coprime to $p$ the inequality $f(p) \leq S_f(m)+1$ implies $f(p^{k+1}) \leq S_f(mp^k)$. 
\end{thm}

\begin{proof}
Let $p$ be a prime for which $f(p)\leq S_f(m) +1$. Therefore $mp^{k+1}$ should be $f$-practical if $f$ is convenient for every $k\geq 0$. If there is an integer $k$ for which $f(p^{k+1})> S_f(mp^k) +1$ then $S_f(mp^k)+1$ is not representable, which is a contradiction. Therefore, for a convenient function $f$, the inequality $f(p)\leq S_f(m) +1$ always implies that $f(p^{k+1}) \leq S_f(mp^k) +1$ for every integer $k$.

If we have $f(p^{k+1})\leq S_f(mp^k) +1$ for every $k$ and if $f(p)\leq S_f(m) +1$ holds for coprime $p$ and $m$, we can use Theorem \ref{thm:f-practicalconstruct} to show that $mp^k$ is $f$-practical for every $k$. Therefore every integer $n$ which fulfills the conditions of Definition \ref{def:convenient} is $f$-practical. Since it has already been shown that every $f$-practical fulfills this condition, we have that $f$ is convenient.
\end{proof}

\begin{thm}
A function $f$ is convenient if and only if for every prime $p$ for which there is a coprime integer $m$ with $f(p)\leq S_f(m) +1$ the inequality 
\begin{align*}
    f(p^{k+1}) \leq f(p)f(p^k)
\end{align*}
holds.
\end{thm}

\begin{proof}
Let $p$ be a prime and $m\in\N$ with $\gcd(m,p)=1$ and $f(p)\leq S_f(m) +1$. Assume that the above inequality holds for $p$. We then show $f(p^{k+1})\leq S_f(mp^k) +1$ by induction over $k$. The base case is fulfilled for $k=0$ since we have $f(p) \leq S_f(m)+1$. Assume that $f(p^{i+1})<S_f(mp^i)+1$ for all $i<k$. We obtain 
\begin{align*}
    f(p^{k+1})&\leq f(p)f(p^k) \leq  (S_f(m)+1)f(p^k) = S_f(m)f(p^k) + f(p^k) \\
    &\leq S_f(m) f(p^k) + S_f(mp^{k-1}) +1 = S_f(m)f(p^k) + S_f(m)S_f(p^{k-1}) +1 \\
    &= S_f(m)(S_f(p^{k-1} +f(p^k))) +1 = S_f(mp^k) +1.
\end{align*}
Therefore, we have $f(p^{k+1}) \leq f(p)f(p^k)$ for some $p$ for which there is an integer $m$ coprime to $p$ with $f(p)\leq S_f(m) +1$, so $f$ is convenient.

Assume $f$ is convenient. Hence, for all primes $p$ and integers $m$ coprime to $p$ with $f(p)\leq S_f(m)+1$, we have $f(p^{k+1}) \leq S_f(mp^k)+1$ for every $k$. For every such $p$ we also have
\begin{align*}
    f(p^{k+1}) &\leq S_f(mp^k) +1 = S_f(m)S_f(p^k) +1 = S_f(m)(f(p^k) + S_f(p^{k-1})) +1 \\
    &= f(p^k)S_f(m) + S_f(p^{k-1}m)+1 \leq f(p^k)(f(p)-1) +S_f(mp^{k-1}) +1 \\
    &= f(p^k)f(p) +S_f(mp^{k-1}) +1 -f(p^k)\\
    &\leq f(p^k)f(p)
\end{align*}
Therefore every convenient $f$ fulfills the above condition.
\end{proof}

\begin{cor}
Let $f$ be convenient. Suppose there is at least one prime $p$ with $1\leq f(p)\leq 2$. Then, for every $k$ primes $f(p_1) \leq f(p_2) \leq \dots \leq f(p_k)$ where there is at least one prime $p_j$ with $1\leq f(p_i)\leq 2$ for all $i \leq j$, there are $k$ integers $E_1,E_2,\dots,E_k$ so that, for every $k$ integers $e_1,\dots,e_k$ with $e_i\geq E_i$, the integer $p_1^{e_1}\dots p_k^{e_k}$ is $f$-practical.
\end{cor}

\section{$f$-practical sets with various densities}\label{densities}

As we remarked in the introduction, the set of practical numbers and the set of $\varphi$-practical numbers both have asymptotic density $0$. In this section, we examine the densities of other $f$-practical sets. First, we give some natural examples with asymptotic density $1$.

\begin{example}
Let $\tau$ be the count-of-divisors function. Every positive integer is $\tau$-practical. This follows from Theorem \ref{thm:everyint} due to the fact that $\tau(p^k) = k+1 \leq \frac{k(k+1)}{2}+1 = S_f(p^{k-1}) +1 \leq S_f(n)S_f(p^{k-1}) +1 = S_f(np^{k-1}) +1$ holds for every prime $p$ and positive integer $k$.
\end{example}

If we take the inequality \eqref{eq:everyint} as an equality for every $p$ and $k$ we obtain following function.

\begin{example}
Let $v_p(n)$ denote the $p$-adic valuation of $n$. The function $h:\N\to\N$ is defined by 
\begin{align*}
    h(n) = 2^{\sum_p v_p(n)}.
\end{align*}
Since we have $v_p(ab) = v_p(a) +v_p(b)$ this function is multiplicative. It satisfies the condition of Theorem \ref{thm:everyint} and therefore every positive integer is $h$-practical.
\end{example}




The following lemma shows that one can construct $f$-practical sets with any density. 

\begin{lemma}\label{thm:denseconstruction} 
For each $n\in \N$, there is a function $f_n$ such that the asymptotic density of $f_n$-practical numbers in $\N$ is $1-\frac{\phi(n)}{n}$. 
\end{lemma}
\begin{proof}
We define the multiplicative function $f_n$ by $f_n(1) = 1$ and
\begin{align*}
    f_n(p^k) = \begin{cases}
        2 &\text{ if } p|n\\
        3 &\text{ else } 
    \end{cases}.
\end{align*}
By Definition \ref{def:iffconvenient}, this function is convenient. So, by definition, the $f_n$-practical numbers are exactly $1$ and the natural numbers divisible by a prime which also divides $n$, since $f(q) = 3 \leq S_f(m) +1$ for every prime $q$ which does not divide $n$ and every $m>1$. Therefore the density of the $f_n$-practical numbers is the density of the numbers not coprime to $n$, which is $1-\frac{\phi(n)}{n}$.

\end{proof}

From this lemma, we can deduce the following theorem:

\begin{thm} The densities of the $f$-practical sets are dense in $[0, 1].$  
\end{thm} 

\begin{proof} From Lemma \ref{thm:denseconstruction}, for any integer $n \in \N$, we can construct a set of $f$-practical numbers with density $1 - \frac{\varphi(n)}{n}$. By \cite[\S 5.17]{Schoenberg}, the values of $\frac{\varphi(n)}{n}$ are dense in $[0, 1].$ Thus, the complementary values $1 - \frac{\varphi(n)}{n}$ must be dense in $[0, 1]$ as well. \end{proof}

\section{Chebyshev bounds for certain $f$-practical sets}\label{boundsforcertainf}
In this section, we demonstrate how the machinery developed in \cite{Thompson} can be used to prove Chebyshev-type bounds for other $f$-practical sets with $f(p^k) \approx (f(p))^k$. We investigate two particular examples with this property: $f = \varphi^*, \lambda.$


\subsection{The $\phi^*$ function}

A divisor $d$ of an integer $n$ is unitary if $\gcd(d, n/d) = 1.$ The unitary totient function $\phi^*$ counts the number of positive integers $k\leq n$ for which the greatest unitary divisor of $n$ which is also a divisor of $k$ is $1$. Therefore we have for all prime $p$ and all integers $k\geq 1$
\begin{align*}
    \phi^*(p^k) = p^k-1.
\end{align*}
Following the second author's proofs in \cite{thesis} we can now establish an upper bound for the number of $\phi^*$-practical integers.

\begin{lemma}\label{lem:evenphistar}
Every even weakly $\phi^*$-practical number is practical.
\end{lemma}

\begin{proof}
Let $n$ be an even $\phi^*$-practical number. Then, with the notation of Definition \ref{def:weaklyfpractical}, for every $0\leq i < \omega(n)$, the inequality $\phi^*(p_{i+1}) = p_{i+1}^{e_i+1}-1\leq 1+ S_{\phi^*}(m_i)$ holds. For every integer $k>1$, we have $\phi^*(k)< k$ and $S_{\phi^*}(k) < \sigma(k)$, since the inequalities hold for the prime powers and the functions are multiplicative. Therefore, for every $m_i>1$, we have $p_{i+1}^{e_{i+1}} = \phi^*(p_{i+1}^{e_{i+1}}) +1 \leq S_{\phi^*}(m_i) + 2 \leq \sigma(m_i) +1$. Since $n$ is even, we have $m_0=1$ and $m_1=2$. Hence, we obtain by induction over $i$ and by \cite[Theorem 1]{Stewart} that $n$ is practical.
\end{proof}

Now we can use Saias' upper bound for the number $PR(X)$ of practical numbers less than or equal to $X$ (\cite[Th\'eor\`eme 2]{Saias}) to prove the next result following the second author's proof in \cite[Theorem 4.8]{thesis}.

\begin{thm}\label{thm:upperphistar}
There exists a positive constant $u_{\phi^*}$ such that 
\begin{align*}
    F_{\phi^*}(X) \leq u_{\phi^*}\frac{X}{\log X}
\end{align*}
holds for any $X\geq 2$.
\end{thm}

\begin{proof}
Let $
X$ be a positive number and let $n$ be a $\phi^*$-practical number in the interval $(0,X]$. Therefore $n$ is also weakly $f$-practical. If $n$ is even it is also practical. If $n$ is odd there is a unique integer $l$ so that $2^ln$ is in $(X,2X]$. Corollary \ref{cor:smallprime} implies that $2^ln$ is also weakly $\phi^*$-practical for $n>1$ and it is easy to see that every power of 2 is weakly $\phi^*$-practical. Thus, we obtain
\begin{align*}
    F_{\phi^*} (X) &= \#\{n\leq X: n \text{ is } \phi^*\text{-practical}\}\\
    &= \#\{n\leq X: n \text{ is even and } \phi^*\text{-practical}\} +\#\{n\leq X: n \text{ is odd and } \phi^*\text{-practical}\}\\
    &\leq \#\{n\leq X: n \text{ is practical}\} + \#\{X<m\leq 2X: m \text{ is practical}\}\\
    &= \mathrm{PR}(2X).
\end{align*}
As proven by Saias \cite[Theorem 2]{Saias} there exists a positive constant $u_{I}$ so that
\begin{align*}
    \mathrm{PR}(X) \leq u_I \frac{X}{\log X}
\end{align*}
holds. By choosing $u_{\phi^*}=2u_I$ we obtain the desired result.
\end{proof}
Since we have $\phi(p) = \phi^*(p)$ for every prime $p$ the squarefree $\phi^*$-practical numbers are exactly the squarefree $\phi$-practical numbers. As shown by the second author in \cite[Lemma 4.17 and Theorem 4.21]{thesis} there exists a lower bound $c\frac{X}{\log X}$ for the number of squarefree $\phi$-practical numbers less than or equal to $X$. Since the squarefree $\phi$-practical and $\phi^*$-practical numbers are the same, we thereby obtain a lower bound for the number of squarefree $\phi^*$-practical numbers less than or equal to $X$.

\begin{thm}
There exists a positive constant $l_{\phi^*}$ so that
\begin{align*}
    l_{\phi^*} \frac{X}{\log X} \leq \#\{n\leq X: n \text{ is squarefree and $\phi^*$-practical}\} \leq F_{\phi^*}(X)
\end{align*}
holds for every $X\geq 2$.
\end{thm}

\subsection{The Carmichael $\lambda$ function}
The Carmichael function $\lambda(n)$ denotes the least integer $m$ for which we have
\begin{align*}
    a^m \equiv 1 \mod n
\end{align*}
for every integer $a$ coprime to $n$. We will use $\lambda^\star$ to denote the set of positive integers that are $f$-practical when $f = \lambda$. We use the $\star$ notation to emphasize that this notion of ``$\lambda$-practical'' differs from the definition of $\lambda$-practical given by the second author in \cite{lambda}, which can be stated as follows: 

\begin{definition}\label{def:lambdapractical} An integer $n$ is $\lambda$-practical if and only if we can write every $m$ with $1 \leq m \leq n$ in the form $m = \sum_{d \mid n} \lambda(d) m_d$, where $m_d$ is an integer with $0 \leq m_d \leq \frac{\varphi(d)}{\lambda(d)}.$ \end{definition}

The values of $n$ satisfying this definition of $\lambda$-practical are precisely those for which the polynomial $x^n-1$ has a divisor of every degree between $1$ and $n$ in $\mathbb{F}_p[x]$ for all primes $p$. The sets of $\lambda$-practical numbers and $\lambda^\star$-practical numbers do not coincide. For example, $156$ satisfies the definition of  $\lambda$-practical in \cite{lambda} but it does not satisfy our definition of $\lambda^\star$-practical. However, it turns out that every $\lambda^\star$-practical number is $\lambda$-practical. We will prove a slightly more general theorem.

\begin{thm}\label{sortedlist} Suppose that $w_1,...,w_t$ and $u_1,...,u_t$ are positive integers, with $w_1 > w_2 > \cdots > w_t$. Let $\mathcal{S} = \sum_{i=1}^t w_i$ and $\mathcal{T} = \sum_{i=1}^t u_i w_i$. Suppose that each positive integer up to $S$ is a subset sum of $w_i$'s. Then each $m \leq \mathcal{T}$ can be written in the form $$m = \sum_{i=1}^t a_i w_i,$$ where $0 \leq a_i \leq u_i$.\end{thm}

\begin{proof} Let $\mathcal{W}$ be a list of $w_i$'s, with $u_i$ instances of each $w_i$, written in decreasing order. Let $k \leq \mathcal{T}$. Starting with the first entry, iteratively subtract elements of $\mathcal{W}$ from $k$, removing each element from the list after it is subtracted to create a new list with one fewer entry. Terminate the process upon arriving at some $k'$ that is either $0$ or smaller than the largest remaining $w_i$, which we will denote $w_j$. By hypothesis, since $k' < w_j \leq S$ then $k'$ is a subset sum of $w_{j+1},...,w_t$. Now, if we add $k'$ to all of the $w_i$'s which were previously subtracted, then $k$ is representable as
\begin{align*}
k=\sum_{i=1}^t a_i w_i,
\end{align*}
with $0 \leq a_i \leq u_i$, as claimed.

\end{proof}

\begin{cor}
Every $\lambda^\star$-practical number is also $\lambda$-practical.
\end{cor}

\begin{proof}
The result follows from applying Theorem \ref{sortedlist} with $t = \tau(n)$; $w_1,...,w_t$ the sorted list of values of $\lambda(d)$ with $d \mid n$; $u_1,...,u_t$ the corresponding values of $\varphi(d)/\lambda(d)$; $\mathcal{S} = S_\lambda(n)$; and $\mathcal{T} = n$. \end{proof}

We can use the upper bound for the count of $\lambda$-practical numbers given by \cite[Proposition 5.1]{lambda} to deduce the following theorem.

\begin{thm}
    There exists a positive constant $u_{\lambda^\star}$ such that 
    \begin{align*}
        F_{\lambda^\star}(X) \leq u_{\lambda^\star} \frac{X}{\log X}
    \end{align*}
    holds.
\end{thm}

Unfortunately, we have been unable to prove a reasonable lower bound for $F_{\lambda^\star}(X)$. It is not clear from our computations (see Tables \ref{table:lambdastar1} and \ref{table:lambdastar2}) whether $X/\log X$ is the correct order of magnitude for the $\lambda^\star$-practicals.


\begin{table}
\centering
\begin{minipage}{0.4\textwidth}
    \begin{tabular}{ | l | l | c |}
    \hline
    $X$ & $F_\lambda^\star(X)$ & $F_\lambda^\star(X)/(X/\log X)$ \\ \hline
    $10^1$ & 6 & 1.381551   \\
    $10^2$ & 28 & 1.289448 \\
    $10^3$ & 164 & 1.132872 \\
    $10^4$ & 1015 & 0.934850 \\
    $10^5$ & 7128 & 0.820641 \\
    $10^6$ & 52326 & 0.722910 \\
    $10^7$ & 409714 & 0.660381 \\
    \hline
    \end{tabular}
    \caption{Ratios for $\lambda^\star$-practicals}\label{table:lambdastar1}
    \end{minipage}
    \hspace{0.4 in}
    \begin{minipage}{.5\textwidth}
    \begin{tabular}{ | l | l | c | }
    \hline
    $X$ & $F_\lambda^\star(X)$ & $F_\lambda^\star(X)/(X/\log X)$\\ \hline
    $1 \cdot 10^6$ & 52326 & 0.722910\\
    $2 \cdot 10^6$ & 96667 & 0.701254\\
    $3 \cdot 10^6$ & 139139 & 0.691712\\
    $4 \cdot 10^6$ & 179854 & 0.683526\\
    $5 \cdot 10^6$ & 219598 & 0.677458\\
    $6 \cdot 10^6$ & 258656 & 0.672819\\
    $7 \cdot 10^6$ & 297202 & 0.669189\\
    $8 \cdot 10^6$ & 335181 & 0.665961\\
    $9 \cdot 10^6$ & 372779 & 0.663246\\
    $1 \cdot 10^7$ & 409714 & 0.660381\\
    \hline
    \end{tabular}\caption{A closer look at the range from $10^6$ to $10^7$}\label{table:lambdastar2}
\end{minipage}\hfill
\end{table}



\section{$f$-practicals for non-multiplicative $f$}\label{non-multiplicativef}

In this section, we remove the condition that $f$ is multiplicative and study the corresponding $f$-practical sets for several well-known non-multiplicative functions. 


\subsection{additive functions}
The naive criterion of Proposition \ref{thm:thesisA1} is of much better use for additive functions than for multiplicative functions. For the additive functions we want to consider, we require $f(p)\geq 1$ for every prime $p$.

\begin{lemma}\label{lem:naiveadditive}
Let $n=\prod_{i=1}^k p_i^{e_i}$ be a positive integer with prime $p_i$. Then $n$ is $f$-practical for an additive function $f$ if and only if 
\begin{align*}
    f(p_i^e) \leq 1 + \sum_{\substack {d|n\\f(d) < f(p_i^e)}} f(d)
\end{align*}
holds for every $1\leq i\leq k$ and $e\leq e_i$.
\end{lemma}

\begin{proof}
It follows immediately from Proposition \ref{thm:thesisA1} that this inequality is necessary for $n$ to be $f$-practical. Now let $t = \prod_{i=1}^k p_i^{a_i}$ be a divisor of $n$. Since $f$ is additive we have $f(t) > f(p_i^{a_i})$ for every $i$ if $t$ is not a prime power, i.e., we must have $a_i>0$ for at least two different $i$. We thereby obtain
\begin{align*}
    f(t) = f(\prod_{i=1}^k p_i^{a_i}) = \sum_{i=1}^k f(p_i^{a_i}) \leq 1 + \sum_{\substack{d|n\\f(d) < f(t)}} f(d)
\end{align*}
which implies that $n$ is $f$-practical.
\end{proof}

\begin{cor}\label{cor:additiveconvenient}
For additive functions $f$, every integer $n$ is $f$-practical if and only if $f(p^k) \leq 1+\sum_{i=0}^{k-1} f(p^i)$ holds for every prime $p$ and any positive integer $i$.
\end{cor}

\begin{remark} Let $\omega(n)$ denote the number of distinct prime factors of an integer $n$, and let $\Omega(n)$ denote the number of prime factors of $n$ with multiplicity. Both functions are additive but not multiplicative. Corollary \ref{cor:additiveconvenient} shows that every positive integer $n$ is $\omega$-practical and $\Omega$-practical. \end{remark}

\begin{remark} Let $f = v_p(n)$, the exact power of $p$ dividing $n$. The fact that the set of $f$-practicals encompasses all natural numbers follows from Corollary \ref{cor:additiveconvenient}. One can also prove that all natural numbers are $v_p$-practical via a simple combinatorial argument: we can write $$S_{v_p}(n) = \frac{{v_p(n)}(v_p(n) + 1)}{2} \cdot \tau\left(\frac{n}{p^{v_p(n)}}\right),$$ where $\frac{{v_p(n)}(v_p(n) + 1)}{2} = 1 + 2 + \cdots + v_p(n)$ is the sum of all of the valuations at powers of $p$ and $\tau(\frac{n}{p^{v_p(n)}})$ represents the number of identical copies of the valuations $1,2,...,v_p(n)$, which come from multiplying the powers of $p$ by each of the divisors of $n$ that are coprime to $p$. Every integer $m$ in the interval $[1, S_{v_p}(n)]$ can be represented as $m = v_p(n) q + r$ for some $r, q$ satisfying $0 \leq r < v_p(n)$ and $0 \leq q \leq \tau(n/p^{v_p(n)})$. \end{remark}

As the next example demonstrates, there are also some natural examples of additive functions for which the set of $f$-practicals does not coincide with the full set of natural numbers.

\begin{example}
Let $a_1(n) = \sum_{p|n} p$, the sum of the distinct primes dividing $n$. For every $n>1$ and every $1<d|n$ there is a prime $p>1$ which divides $d$. Therefore we have $a_1(d) \geq p >1$. Hence the number $1< a_1(n) \leq S_{a_1}(n)$ is not representable. In particular, this shows that $1$ is the only $a_1$-practical number. 
\end{example}
\subsection{Functions which are neither additive nor multiplicative}

We can also define $f$-practical numbers for functions $f$ which are neither multiplicative nor additive. One such example is the sum-of-proper-divisors function, which is defined as follows:

\begin{definition} Let $s: \N \to \N$ be given by $s(n)= \sigma(n) - n.$\end{definition}

The function $s$ is used in the study of perfect numbers. Namely, if $s(n) = n$ then $n$ is perfect. If $s(n) > n$, we say that $n$ is abundant. Since $s$ is an arithmetic function, we can use the $f$-practical definition to define $s$-practical numbers. We begin by demonstrating that there are infinitely many $s$-practical numbers. To show this, we will need the following lemma.

\begin{lemma}\label{lem:calcs}
For two coprime integers $a,b$ we have $$s(ab) = s(a)s(b) + as(b) + bs(a).$$ 
\end{lemma}

\begin{proof} Observe that
\begin{align*} 
    s(ab) &= \sigma(ab) -ab = \sigma(a)\sigma(b) -ab = (s(a)+a)(s(b) +b) -ab \\
    &= s(a)s(b) + as(b) + bs(a).
\end{align*}
\end{proof}

\begin{thm} There are infinitely many $s$-practical numbers.\end{thm}

\begin{proof} 
Every prime is $s$-practical, since we have 
\begin{align*}
    S_s(p) = \sum_{d|p}s(d) = s(1) + s(p) = 0 + 1= 1
\end{align*}
for prime $p$.
\end{proof}

The function $s$ is not multiplicative, which prevents us from using the machinery developed in previous sections to prove upper and lower bounds for the count of $s$-practical numbers. However, it is still possible to show that the $s$-practical numbers arising from integers $n$ with $n \leq 2S_s(n)$ have asymptotic density $0$. To prove this, we will follow an argument that was used by the second author in \cite{Thompson} to show that the $\varphi$-practical numbers have asymptotic density $0$. We note that Erd\H{o}s \cite{Erdos} was the first to claim that the practical numbers have asymptotic density $0$. Although he did not write down a proof, it is likely that he had a similar argument in mind. 

\begin{thm}\label{thm:abundantdensity}
If $n^{1/2} \leq S_s(n)$ then the $s$-practicals have asymptotic density $0$.
\end{thm}

\begin{proof}
We have $ \tau(n)\leq 2^{\Omega(n)}$. Because  $\Omega(n)$ has normal order $\log\log n$, for all $n$ except a set with asymptotic density $0$ we have
\begin{align*}
    \tau(n)\leq 2^{\Omega(n)} \leq 2^{(1+\epsilon)\log\log(n)} = (\log n)^{(1+\epsilon)\log 2} \leq (\log n)^{0.7}
\end{align*}
if we fix $\epsilon=1/1000$. But for every $s$-practical $n$ it has to be the case that $ S_s(n)\leq 2^{\tau(n)-1}$ since there are at most $2^{\tau(n)}$ different numbers which can be represented as the sum of $s(d)$'s where the $d$'s are some of the $\tau(n)$ divisors of $n$ and each number between $1$ and $S_s(n)$ has to be representable as such a sum in order for $n$ to be $s$-practical. Because we have $s(1)=0$, half of these possible sums coincide, as there is no difference between the sums with and without $s(1)$. From the hypothesis $n^{1/2} \leq S_s(n)$, we hence obtain
\begin{align*}
    \frac{1}{2} \log n \leq \log S_s(n)\leq \tau(n)\log 2 < \tau(n) \leq (\log n)^{0.7}.
\end{align*}
But for all $n \geq e^{8\sqrt[3]{2}}$, we have
\begin{align*}
    \frac 12(\log n)^{0.3} \geq \frac12 \left(\sqrt[3]{1024}\right)^{0.3} =  \frac 12 \cdot 1024^{0.1} = 1.
\end{align*}
So for almost all $n$ the inequality $\frac12\log n \geq (\log n)^{0.7}$ holds. Therefore the set of $s$-practical values of $n$ with $n^{1/2} \leq S_s(n)$ has asymptotic density 0.
\end{proof}

\begin{lemma}\label{lemma:density1ineq} The inequality $n^{1/2} \leq S_s(n)$ holds for almost all $n$. \end{lemma}

\begin{proof} If $n$ is composite, its least prime factor $p$ satisfies $p \leq n^{1/2}$, so $n/p$ is a proper divisor of $n$ that is $\geq n^{1/2}$. Thus, we have $$S_s(n) \geq s(n) \geq \frac{n}{p} \geq n^{1/2}.$$ Since the set of composite numbers has asymptotic density $1$, it follows that $S_s(n) \geq n^{1/2}$ holds for almost all $n$. \end{proof}

We can deduce the following corollary from Theorem \ref{thm:abundantdensity} and Lemma \ref{lemma:density1ineq}.

\begin{cor} The set of $s$-practical numbers has asymptotic density $0$. \end{cor}

\section*{Acknowledgements}

This research was initiated when the first author was a student in the intern program at the Max-Planck-Institut f\"{u}r Mathematik and while the second author was a visiting researcher there. Both authors would like to thank the Max-Planck-Institut f\"{u}r Mathematik for making this collaboration possible. Portions of this work were completed while the second author was in residence at the Mathematical Sciences Research Institute, during which time she was supported by the National Science Foundation under Grant No. DMS-1440140. The second author is also supported by an AMS Simons Travel Grant. Both authors are grateful to Carl Pomerance for posing the question answered in Theorem \ref{thm:densitiesaredense} and for suggesting the generalized version of Theorem \ref{sortedlist} that is presented in this paper. The authors are also grateful to Paul  Pollack for helpful comments which led to an improvement in their computation of the asymptotic density of $s$-practicals. 


\providecommand{\bysame}{\leavevmode\hbox
to3em{\hrulefill}\thinspace}
\providecommand{\MR}{\relax\ifhmode\unskip\space\fi MR }
\providecommand{\nMRhref}[2]{%
  \href{http://www.ams.org/mathscinet-getitem?mr=#1}{#2}
} \providecommand{\href}[2]{#2}

\end{document}